\documentclass[12pt,a4paper,final]{article}
\usepackage{amsmath}%
\usepackage{amstext}%
\usepackage{amssymb}%
\usepackage{showkeys}%
\usepackage{epsfig}%
\usepackage{graphicx}%
\usepackage{xcolor}
\usepackage{multicol}
\usepackage{amsthm}
\usepackage[symbol]{footmisc}

\usepackage[top=1in, bottom=1.25in, left=0.8in, right=0.8in]{geometry}

\newtheorem{prop}{Proposition}
\newtheorem{defn}{Definition}
\newtheorem{rem}{Remark}
\newtheorem{theorem}{Theorem}[section]

\newtheorem{lemma}[theorem]{Lemma}
\newtheorem{example}{Example}

\newcommand{\R}{\mathbb{R}}%
\newcommand{\N}{\mathbb{N}}%
\newcommand{\Hi}{\mathcal{H}}%

\author{Lkhamsuren Altangerel\footnote{German-Mongolian Institute for Resources and Technology, Nalaikh, Mongolia, E-mail: altangerel@gmit.edu.mn}}

\date{}

\title{Tseng's Type Methods in Continuous and Discrete Time for Quasi-Variational Inequalities\footnote{This work was partially
 supported by a grant from the IMU-CDC and Simons Foundation}}

\begin{document}
\maketitle

\begin{abstract}
    This paper presents an approach for obtaining approximate solutions to quasi-variational inequalities in a real Hilbert space by modifying Tseng's scheme, which was originally designed for variational inequalities. The study explores the existence of equilibrium points and investigates convergence results related to dynamical systems. Linear convergence for discretized systems is examined through  examples, illustrations, and special cases.  
\end{abstract}

\section{Introduction}
Quasi-variational inequalities represent a significant extension of classical variational inequalities, where the constraint set depends on the solution itself.  They were initially introduced by Bensoussan et al. \cite{bensoussan1973controle} in the context of impulse control problems and have since evolved into powerful mathematical tools with applications in economics, engineering, operations research, and partial differential equations. 

A standard sufficient condition for existence and uniqueness under Lipschitz continuous strongly monotone mapping was provided in \cite{noor1994general}. For further study, we refer to \cite{barbagallo2024existence,dreves2016uniqueness,nesterov2006solving,kovcvara1995class}.

There have been some results regarding the existence and uniqueness of solutions for quasi-variational inequalities under weaker assumptions, such as strongly pseudomonotonicity \cite{nguyen2020some,van2023existence}. Additionally, the relationships with other models such as the generalized Nash equilibrium problems \cite{dreves2016uniqueness}, and multi-leader follower games \cite{pang2005quasi} have been investigated. Regarding the applications of quasi-variational inequalities we refer to \cite{scrimali2007quasi} in telecommunication networks, \cite{scrimali2004quasi} in transportation networks, and \cite{lenzen2014solving} in image restoration. 

Recent advances in computational mathematics have led to the development of more efficient algorithms for solving QVIs. Antipin et al. \cite{antipin2018extragradient} (also see \cite{mijajlovic2023strong}) presented gradient projection and extragradient methods for solving QVIs under the assumptions of strong monotonicity and Lipschitz continuity of associated mappings.
Others include inertial projection-type methods \cite{keten2021efficient,shehu2020inertial,shehu2022linear,yao2024linear}, C-FISTA type projection algorithm \cite{yao2025c}, and gradient-type projection methods \cite{mijajlovic2019gradient}. 

Motivated by the above results, we introduce a new approach for solving QVI based on Tseng's scheme (\cite{tseng2000modified}, also see \cite{boct2020forward}), which has been investigated for variational inequality problems. Compared to existing algorithms, the main contributions of the paper can be highlighted as follows: 

\begin{itemize}
    \item To the best of our knowledge the Tseng scheme for variational inequalities has been extended to quasi-variational inequalities, sparking broad research interest in this direction.
    \item Continuous and discrete time algorithms have been investigated for quasi-variational inequalities in the general case with set-valued mappings.
    \item Similar to the inertial projection-type methods, the projection and mapping values will be evaluated once in each iteration.
\end{itemize}

The outline of the paper is as follows. Section 2 will review the definitions and preliminary results used in this paper. In Section 3, we will present the existence of equilibrium points and convergence results related to dynamical systems. Section 4 is dedicated to examining  linear convergence for discretized systems, with some examples, illustrations, and special cases included. The paper concludes with a mention of research ideas that can be considered in the future.

\section{Preliminaries}
Let $\Hi$ be a real Hilbert space with inner product $\langle\cdot,\cdot\rangle$, 
and topology induced by the norm $\|x|=\sqrt{\langle x,x\rangle}$.
Let $F:\Hi\rightarrow \Hi$ a nonlinear operator, and $K:\Hi \rightrightarrows \Hi$ be a set-valued mapping such that for each $x\in \Hi,~K(x)$ is a nonempty closed convex subset of $\Hi.$ The Quasi-Variational Inequality consists of finding $x^*\in K(x^*)$ such that 
$$
     (QVI)\hspace{1cm}\langle F(x^*),y-x^*\rangle \geq 0~~\forall y\in K(x^*).
$$
\begin{rem} In the case of a single set $K$, the problem $(QVI)$ reduces to the classical variational inequality of finding $x^*\in K$
$$
(VI)\hspace{1cm}\langle F(x^*),y-x^*\rangle \geq 0~~ \forall y\in K.
$$
\end{rem}

\begin{defn}
    Let $F:\Hi\rightarrow \Hi$ be a given mapping. The mapping $F$ is said to be:
    \begin{itemize}
        \item[(i)] $L$-Lipschitz continuous with $L>0$ if 
 $$\|F(x)-F(y)\|\leq L\|x-y\|~~\forall x,y\in \Hi.$$
        \item[(ii)] monotone if
         $$\langle x-y,F(x)-F(y)\rangle\geq 0~~\forall x,y\in \Hi.$$
         \item[(iii)] $\rho$-strongly monotone if there exists $\rho> 0$ such 
 $$\langle x-y,F(x)-F(y)\rangle\geq \rho\|x-y\|^2~~\forall x,y\in \Hi.$$ 
    \end{itemize}
\end{defn}
The projection $P_C :\Hi\to C$ onto the nonempty, closed, convex set $C\subseteq \Hi$ is defined by
$$P_C(x)=\mbox{argmin}\{\|x-y\|:~y\in C\}.$$
Note that for each $x\in \Hi,$ the implicit projection $P_{K(x)}$ is nonexpensive, i.e.,
$$\|P_{K(x)}(u)-P_{K(x)}(v)\|\leq \|u-v\|~~\forall u,v\in \Hi.$$

\begin{prop} \cite{noor1994general} \label{exist}
     Let $F:\Hi\rightarrow \Hi$ be $L$-Lipschitz continuous and $\rho$-strongly monotone on $\Hi.$ Assume that $K:\Hi \rightrightarrows \Hi$ is a set-valued mapping with nonempty, closed and convex values. Setting $\gamma:=\frac{L}{\rho}\geq 1$ and assume that there is some constant $0\leq l\leq \frac{1}{\gamma(\gamma+\sqrt{\gamma^2-1})}$ such that  
    \begin{equation} \label{cond_proj1}
\| P_{K(x)}(z)-P_{K(y)}(z)\|\leq l\|x-y\|~~\forall x,y,z\in \Hi,
    \end{equation}
 then the quasi-variational inequality $(QVI)$ has a unique solution $x^*.$ 
\end{prop}

\begin{rem}
    In \cite{nesterov2006solving}, it was shown that uniqueness is guaranteed as long as  $0\leq l\leq \frac{1}{\gamma}.$ 
\end{rem}

\begin{lemma}
    Let $K:\Hi \rightrightarrows \Hi$ be a set-valued mapping with nonempty, closed and convex subset in $\Hi.$ $x^*\in K(x^*)$ is solution of $(QVI)$ if and only if for any $\alpha >0$ it holds
    $$x^*=P_{K(x^*)}(x^*-\alpha F(x^*)).$$
\end{lemma}
    
\begin{lemma}
    Let $z\in \Hi.$ The necessary and sufficient characterizations of the projection are  
    $$P_{K(x)}(z)\in K(x)\Leftrightarrow~\langle P_{K(x)}(z)-z,y-P_{K(x)}(z) \rangle \geq 0~~\forall y\in K(x),$$
or, equivalently
$$\langle z-P_{K(x)}(z),y-P_{K(x)}(z) \rangle \leq 0~~\forall y\in K(x).$$ 
\end{lemma}

\begin{lemma}
   Under the assumptions of Proposition \ref{exist} it holds
   \begin{equation}
   \|P_{K(x)}(x-\lambda F(x))-P_{K(y)}(y-\lambda F(y))\|\leq \theta \|x-y\|\quad \forall x,y\in \Hi,    
   \end{equation}
   where $\theta=l+\sqrt{1-2\lambda \rho+\lambda^2 L^2}.$
\end{lemma}

\begin{proof}
    Let $x,y\in \Hi.$ Then, it holds
    \begin{eqnarray*}
    &&\|P_{K(x)}(x-\lambda F(x))-P_{K(y)}(y-\lambda F(y))\|\\
    &\leq& \|P_{K(x)}(x-\lambda F(x))-P_{K(x)}(y-\lambda F(y))\|+\|P_{K(x)}(y-\lambda F(y))-P_{K(y)}(y-\lambda F(y))\|\\
    &\leq& \|x-\lambda F(x)-(y-\lambda F(y))\|+l\|x-y\|.
    \end{eqnarray*}
    On the other hand, we have
    \begin{eqnarray*}
        \|x-\lambda F(x)-(y-\lambda F(y))\|^2&=&\|x-y\|^2+\lambda^2\|F(x)-F(y)\|^2-2\lambda \langle x-y, F(x)-F(y) \rangle\\
        &\leq&(1+\lambda^2 L^2-2\lambda \rho)\|x-y\|^2,
    \end{eqnarray*}
    or, equivalently
    \begin{equation}\label{est_sqrt}
        \|x-\lambda F(x)-(y-\lambda F(y))\|\leq \sqrt{1+\lambda^2 L^2-2\lambda \rho}\|x-y\|.
    \end{equation}
    By considering that $\theta= l+\sqrt{1+\lambda^2 L^2-2\lambda \rho},$ it holds
    \begin{equation}\label{est_theta}
        \|P_{K(x)}(x-\lambda F(x))-P_{K(y)}(y-\lambda F(y))\|\leq \theta \|x-y\|.
    \end{equation}
\end{proof}

\section{Dynamical system of the Tseng type}
Let $T:\Hi\rightarrow \Hi$ be a continuous mapping, and $x:(0,+\infty)\rightarrow \Hi.$
We recall the stability concepts of an equilibrium point of the general dynamical system
\begin{equation}\label{dyn_sys}
    \dot{x}(t)=T(x(t))~~t\geq 0.
\end{equation}

\begin{defn}
    \begin{itemize}
        \item[]
         \item[(i)]  A point $x^*\in \Hi$ is an equilibrium point for \eqref{dyn_sys} if $T(x^*)=0.$ 
    \item[(ii)] An equilibrium point $x^*$ of \eqref{dyn_sys} is stable if, $\forall \varepsilon>0,~\exists \delta>0$ such that, $\forall x_0\in B(x^*,\delta),$ the solution $x(t)$ of the dynamical system with $x(0)=x_0$ exists and 
    $$x(t)\in B(x^*,\varepsilon):=\{y\in \Hi|~\|x^*-y\|<\varepsilon\}~~\forall t> 0.$$
    \item[(iii)]  An equilibrium point $x^*$ of \eqref{dyn_sys} is exponentially stable if $\exists \varepsilon>0$ and constants $k>0,~\delta>0$ such that, for each solution $x(t)$ of \eqref{dyn_sys} with $x(0)\in B(x^*,\varepsilon),$ one has
        $$
          \|x(t)-x^*\|\leq k\|x(0)-x^*\|e^{-\delta t}~~\forall t\geq 0.  
       $$
    \end{itemize}
\end{defn}

We consider the following dynamical system of the Tseng type
\begin{eqnarray*} 
    &&y(t)=P_{K(x(t))}(x(t)-\lambda F(x(t)))\\
    &&\dot{x}(t)+x(t)=y(t)+\lambda (F(x(t))-F(y(t)))\\
    &&x(0)=x_0,
\end{eqnarray*}
where $\lambda >0$ and $x_0\in \Hi.$
The system is reduced to a dynamical system of type
$\dot{x}=f(x),$ where
\begin{equation} \label{Tseng_dynsys}
    f(x)=P_{K(x)}(x-\lambda F(x))+\lambda (F(x)-F(P_{K(x)}(x-\lambda F(x)))-x.
\end{equation}

\begin{theorem}
     Let the assumptions of Proposition \ref{exist} be fulfilled. Then the dynamical system \eqref{Tseng_dynsys} has an unique solution.
\end{theorem}

\begin{proof}
\begin{eqnarray*}
    &&\|f(x)-f(z)\|\leq \|z-x\|+\|
    P_{K(z)}(z-\lambda F(z))-P_{K(x)}(x-\lambda F(x))\|\\
    &+&\lambda \|F(z)-F(x)\|+\lambda\|F(P_{K(z)}(z-\lambda F(z))-F(P_{K(x)}(x-\lambda F(x))\|\\
    &\leq& (1+\lambda L)\|z-x\|+(1+\lambda L)\|
    P_{K(z)}(z-\lambda F(z))-P_{K(x)}(x-\lambda F(x))\|\\
    &\leq& (1+\lambda L)\|z-x\|+(1+\lambda L)\|
    P_{K(z)}(z-\lambda F(z))-P_{K(x)}(z-\lambda F(z))\|\\
    &+&(1+\lambda L)\|
    P_{K(x)}(z-\lambda F(z))-P_{K(x)}(x-\lambda F(x))\|\\
    &\leq& (1+l)(1+\lambda L)\|z-x\|+(1+\lambda L)\|z-\lambda F(z)-(x-\lambda F(x)\|.
\end{eqnarray*}
Using \eqref{est_sqrt}, it becomes
\begin{equation} \label{Lipshitz_f}
    \|f(x)-f(z)\|\leq (1+\lambda L)(1+l+\sqrt{1-2\rho\lambda+\lambda^2L^2})\|z-x\|=(1+\lambda L)(1+\theta)\|z-x\|,
\end{equation}
which shows that $f$ is a Lipchitz continues, and the existence and uniqueness follows from the classical Cauchy-Lipschitz-Picard theorem (\cite{feng2016note}). 
\end{proof}

\begin{theorem}
      Let the assumptions of Proposition \ref{exist} be fulfilled.  Assume that
     $$(1+\lambda L)(1+\theta)<2.$$
     Then the dynamical system \eqref{Tseng_dynsys} converges to the solution of the $(QVI)$ at the rate 
     $$\|x(t)-x^*\|\leq \|x_0-x^*\|e^{\Lambda t},$$
where $x^*$ is the equilibrium point of the system.
\end{theorem}

\begin{proof}
    We consider the Lyapunov function $V(x(t))=\frac12\|x(t)-x^*\|^2.$ By mentioning that\\ $x^*=P_{K(x^*)}(x^*-\lambda F(x^*)),$ and using \eqref{est_sqrt} it holds
    \begin{eqnarray*}
        \dot{V}(x)&=&\frac{d}{dt} V(x(t))=\langle x(t)-x^*, \dot{x}\rangle\\
        &=&\langle x(t)-x^*, P_{K(x(t))}(x(t)-\lambda F(x(t)))+\lambda (F(x(t))-F(P_{K(x(t))}(x(t)-\lambda F(x(t))))-x(t)\rangle\\
        &=&\langle x(t)-x^*, P_{K(x(t))}(x(t)-\lambda F(x(t)))+\lambda (F(x(t))-F(P_{K(x(t))}(x(t)-\lambda F(x(t))))-x(t)\\
        &-&(P_{K(x^*)}(x^*-\lambda F(x^*))-x^*)\rangle\\
        &=&\langle x(t)-x^*, P_{K(x(t))}(x(t)-\lambda F(x(t))) -P_{K(x^*)}(x(t)-\lambda F(x(t)))\rangle\\
        &+&
        \langle x(t)-x^*, P_{K(x^*)}(x(t)-\lambda F(x(t))) -P_{K(x^*)}(x^*-\lambda F(x^*))\rangle\\
        &+& \lambda \langle  x(t)-x^*,F(x(t))-F(x^*)\rangle\\
        &+& \lambda \langle  x(t)-x^*, F(P_{K(x^*)}(x^*-\lambda F(x^*)))-F(P_{K(x(t))}(x(t)-\lambda F(x(t))))\rangle - \|x(t)-x^*\|^2\\
        &\leq& l \|x(t)-x^*\|^2+\|x(t)-x^*\|\|x(t)-\lambda F(x(t))-(x^*-\lambda F(x^*))\|+\lambda L\|x(t)-x^*\|^2\\
        &+&\lambda L\|x(t)-x^*\|\|P_{K(x^*)}(x^*-\lambda F(x^*))-P_{K(x(t))}(x(t)-\lambda F(x(t)))\|-\|x(t)-x^*\|^2\\
        &\leq&(l+\lambda L+\sqrt{1-2\lambda \rho+\lambda^2 L^2})\|x(t)-x^*\|^2\\
        &+&\lambda L(\|P_{K(x^*)}(x^*-\lambda F(x^*))-P_{K(x^*)}(x(t)-\lambda F(x(t)))\|\\
        &+&\|P_{K(x^*)}(x(t)-\lambda F(x(t)))-P_{K(x(t))}(x(t)-\lambda F(x(t)))\|)\|x(t)-x^*\|-\|x(t)-x^*\|^2\\
        &\leq& (l+\lambda L+\sqrt{1-\lambda \rho+\lambda^2 L^2}+\lambda L(l+\sqrt{1-2\lambda \rho+\lambda^2 L^2})-1)\|x(t)-x^*\|^2\\
        &=& ((1+\lambda L)\theta+\lambda L-1)\|x(t)-x^*\|^2=((1+\lambda L)(1+\theta)-2)V(t).
    \end{eqnarray*}
    By the assumption $\Lambda:=(1+\lambda L)(1+\theta)-2<0$ and it holds
    $$V(t)\leq V_0 e^{\Lambda t},$$ where $V_0=\frac 12\|x-x_0\|^2,$ which shows that the dynamical system converges to the solution of the problem $(QVI).$
\end{proof}

\begin{rem}
    We can demonstrate similar results for the following dynamical system 
     $$f(x(t))=\alpha(t)[P_{K(x(t))}(x(t)-\lambda F(x(t)))+\lambda (F(x(t))-F(P_{K(x(t))}(x(t)-\lambda F(x(t))))-x(t)],
     $$
     $\alpha(t)\in C([0,+\infty))$ and $\int\limits_{t_0}^{\infty} \alpha (t)=+\infty.$
     In this case $\Lambda(t)=\alpha (t)((1+\lambda L)\theta+\lambda L-1).$ 
\end{rem}

\section{Discrete system of the Tseng type}
The numerical scheme of the Tseng type is 
\begin{equation}\label{tseng-qvi}
\left\{
\begin{array}{llll}
y^{k}=P_{K(x^k)}\left(x^k-\lambda F(x^k)\right)\\
\\
x^{k+1}=y^k+\lambda\left(F(x^k)-F(y^k)\right)
\end{array}
\right.
\end{equation}
where $\lambda>0.$  

\begin{rem}
    If $K$ is a single set, then the scheme \eqref{tseng-qvi} is Tseng's algorithm \cite{tseng2000modified} for solving $(VI).$
\end{rem}
 
\begin{theorem}
Let the assumptions of Proposition \ref{exist} be fulfilled.Assume that
$$(1+\theta)(1+\lambda L)<\sqrt{4-l^2+2l}-1.$$
Then the sequence $\{x_n\}$ generated by the Tseng type algorithm converges linearly to the unique solution of $(QVI).$
\end{theorem}

\begin{proof}
Let $x^*\in K(x^*)$ be the solution of $(QVI).$ Taking into account that 
$$f(x^*)=P_{K(x^*)}(x^*-\lambda F(x^*))+\lambda(F(x^*)-F(P_{K(x^*)}(x^*-\lambda F(x^*))))-x^*=0,$$ and setting $z=x^k,~x=x^*$ in \eqref{Lipshitz_f}, we have 

\begin{eqnarray*}
  \|f(x^k)\|&=&\|x^k-P_{K(x^k)}(x^k-\lambda F(x^k)) -\lambda(F(x^k)-F(P_{K(x^k)}(x^k-\lambda F(x^k))))\|\\
&\leq& (1+\theta)(1+\lambda L)\|x^k-x^*\|,  
\end{eqnarray*}
or, equivalently
\begin{equation}\label{est1}
    \|x^k-y^k -\lambda(F(x^k)-F(y^k))\|\leq (1+\theta)(1+\lambda L)\|x^k-x^*\|.
\end{equation}

On the other hand,
\begin{eqnarray*}
   &&\langle x^k-P_{K(x^k)}(x^k-\lambda F(x^k)) -\lambda(F(x^k)-F(P_{K(x^k)}(x^k-\lambda F(x^k)))),x^k-x^*\rangle\\
=&&\langle x^k-P_{K(x^k)}(x^k-\lambda F(x^k)) -\lambda(F(x^k)-F(P_{K(x^k)}(x^k-\lambda F(x^k))))\\
-&&x^*+P_{K(x^*)}(x^*-\lambda F(x^*)) +\lambda(F(x^*)-F(P_{K(x^*)}(x^*-\lambda F(x^*)))),x^k-x^*\rangle\\ 
=&&\|x^k-x^*\|^2-\langle P_{K(x^k)}(x^k-\lambda F(x^k)) - P_{K(x^*)}(x^k-\lambda F(x^k)), x^k-x^*\rangle\\
-&&\langle P_{K(x^*)}(x^k-\lambda F(x^k)) - P_{K(x^*)}(x^*-\lambda F(x^*)), x^k-x^*\rangle -\lambda \langle F(x^k)-F(x^*), x^k-x^*\rangle\\
-&&\lambda \langle F(P_{K(x^*)}(x^*-\lambda F(x^*)))-F(P_{K(x^k)}(x^k-\lambda F(x^k))), x^k-x^*\rangle. 
\end{eqnarray*}
Consequently, taking into account of \eqref{est_sqrt} we have
\begin{eqnarray*}
   &&\langle x^k-P_{K(x^k)}(x^k-\lambda F(x^k)) -\lambda(F(x^k)-F(P_{K(x^k)}(x^k-\lambda F(x^k)))),x^k-x^*\rangle\\
\geq &&\|x^k-x^*\|^2-\frac{1}{2}\|P_{K(x^k)}(x^k-\lambda F(x^k)) - P_{K(x^*)}(x^k-\lambda F(x^k))\|^2-\frac{1}{2}\|x^k-x^*\|^2\\
-&&\|P_{K(x^*)}(x^k-\lambda F(x^k)) - P_{K(x^*)}(x^*-\lambda F(x^*))\|\cdot \|x^k-x^*\| -\lambda L\|x^k-x^*\|^2\\
-&&\lambda L\|P_{K(x^*)}(x^*-\lambda F(x^*))-P_{K(x^k)}(x^k-\lambda F(x^k))\|\cdot\|x^k-x^*\|\\
=&&(1-\frac{1}{2}-\frac{l^2}{2}- \theta+l-\lambda L -\lambda L\theta)\|x^k-x^*\|^2=(\frac{1}{2}-\frac{l^2}{2}- \theta +l -\lambda L -\lambda L\theta)\|x^k-x^*\|^2,
\end{eqnarray*}
or equivalently
\begin{eqnarray}\label{discrete_est2}
    &&\langle x^k- P_{K(x^k)}(x^k-\lambda F(x^k))\nonumber\\
    &-&\lambda(F(x^k)-F(P_{K(x^k)}(x^k-\lambda F(x^k)))),x^k-x^*\rangle\geq \mu \|x^k-x^*\|^2,
\end{eqnarray}
where $\mu:=\frac{1}{2}-\frac{l^2}{2}-\theta+l-\lambda L -\lambda L\theta.$ Assuming $\mu>0,$ and since $y^k=P_{K(x^k)}(x^k-\lambda F(x^k)),$ it can be rewritten as
\begin{equation}\label{mu_yk}
    \langle y^k -x^k +\lambda(F(x^k)-F(y^k)),x^k-x^*\rangle\leq -\mu \|x^k-x^*\|^2
\end{equation}
Combining \eqref{est1} and \eqref{discrete_est2}, one has
\begin{equation} \label{norm^2}
    \|x^k-y^k -\lambda(F(x^k)-F(y^k))\|^2\leq \frac{(1+\theta)^2(1+\lambda L)^2}{\mu} \langle x^k- y^k -\lambda(F(x^k)-F(y^k)),x^k-x^*\rangle.
\end{equation}
By using  \eqref{mu_yk} and \eqref{norm^2}, we obtain that
\begin{eqnarray*}
    \|x^{k+1}-x^*\|^2&=&\|y^k +\lambda(F(x^k)-F(y^k))-x^*\|^2\\
    &=&\|y^k +\lambda(F(x^k)-F(y^k))-x^k\|^2+\|x^k-x^*\|^2\\
    &+&2\langle y^k +\lambda(F(x^k)-F(y^k))-x^k,x^k-x^*\rangle\\
    &\leq&\|x^k-x^*\|^2+\frac{(1+\theta)^2(1+\lambda L)^2}{\mu} \langle x^k- y^k -\lambda(F(x^k)-F(y^k)),x^k-x^*\rangle\\
    &+&2\langle y^k +\lambda(F(x^k)-F(y^k))-x^k,x^k-x^*\rangle\\
    &=&\|x^k-x^*\|^2+\Big(2-\frac{(1+\theta)^2(1+\lambda L)^2}{\mu}\Big)\\
    &\cdot&\langle y^k +\lambda(F(x^k)-F(y^k))-x^k,x^k-x^*\rangle\\
    &\leq &(1-2\mu+(1+\theta)^2(1+\lambda L)^2)\|x^k-x^*\|^2,
\end{eqnarray*}
or equivalently,
\begin{equation} \label{est_k+1}
\|x^{k+1}-x^*\|^2\leq (1-2\mu+(1+\theta)^2(1+\lambda L)^2)\|x^k-x^*\|^2. 
\end{equation}
Finally we get
\begin{equation}
    \|x^{k+1}-x^*\|\leq (1-2\mu+(1+\theta)^2(1+\lambda L)^2)^{\frac{n}{2}}\|x^0 - x^*\|.\end{equation}
It is clear that if 
\begin{equation} \label{disc_est_parameter}
    r:=1-2\mu+(1+\theta)^2(1+\lambda L)^2<1,
\end{equation}
then sequence $\{x^k\}$ converges linearly to the solution of $(QVI).$ 

If the condition \eqref{disc_est_parameter} is fulfilled, then it holds that $\mu>0$ and
\begin{eqnarray}
   && (1+\theta)^2(1+\lambda L)^2<2\mu=2\Big(\frac{1}{2}-\frac{l^2}{2}-\theta+l-\lambda L -\lambda L\theta\Big)\nonumber\\
   \Leftrightarrow&& (1+\theta)^2(1+\lambda L)^2+2(1+\theta)(1+\lambda L)+1<4-l^2+2l\nonumber\label{est_l}\\
   \Leftrightarrow&&((1+\theta)(1+\lambda L)+1)^2<4-l^2+2l.
\end{eqnarray}
We can easy check that $\sqrt{4-l^2+2l}-1>0$ for $0<l<3,$ and consequently, we have
$$(1+\theta)(1+\lambda L)<\sqrt{4-l^2+2l}-1.$$
\end{proof}

\begin{rem}
    Note that the function $\delta(l):=4-l^2+2l$ reaches its maximum value at $l=1.$ Therefore
        $$\sqrt{4-l^2+2l}-1\leq \sqrt{5}-1\approx 1.13.$$
\end{rem}

\begin{example} (cf. \cite{keten2021efficient})\\
We consider the real Hilbert space 
$$\Hi=l_2=\Big\{x=(x_0,x_1,x_2,\dots)|~\sum\limits_{k=0}^{\infty} |x_k|^2<\infty~\mbox{and}~x_k\in \R,~\forall k\in \mathbb{N}_0=\{0,1,2,\dots\}\Big\}$$ endowed with norm induced by inner product $\|x\|_2=\sqrt{\langle x,x\rangle}=\Big(\sum\limits_{k=0}^{\infty} |x_k|^2\Big)^{\frac 12}$ and define an operator $F:\Hi\rightarrow \Hi$ by $F(x)=\alpha x+|\sin x|,$ where $\alpha>1$ is a given constant and $|\sin x|:=(|\sin x_0|, |\sin x_1|,|\sin x_2|,\dots).$ For any $x,y\in \Hi,$ it holds
\begin{eqnarray*}
    \|F(x)-F(y)\|&=&\|\alpha x+|\sin x| - \alpha y-|\sin y| \|\leq \alpha \|x-y\|+\||\sin x|-|\sin y|\|\\
    \leq &&\alpha \|x-y\|+\|x-y\|=(\alpha+1)\|x-y\|.
\end{eqnarray*}
Therefore, $F$ is $(\alpha+1)$-Lipschitz continuous. 
On the other hand
\begin{eqnarray*}
    \langle F(x)+F(y),x-y\rangle &=&\langle 
\alpha x+|\sin x| - \alpha y-|\sin y|, x-y\rangle\\
  &=&\alpha \|x-y\|^2+\langle |\sin x| - |\sin y|,x-y\rangle.
\end{eqnarray*}
According to the mean value theorem, $\exists z_k$ such
that $\sin x_k-\sin y_k=\cos z_k (x_k-y_k),$ and $|\cos z_k|\leq 1.$ Whence, we obtain that $|\sin x_k-\sin y_k|\geq - |x_k-y_k|. $ Consequently, it holds
$$\langle F(x)+F(y),x-y\rangle \geq \alpha \|x-y\|^2 - \|x-y\|^2=(\alpha -1)\|x-y\|^2,$$
which shows that $F$ is $(\alpha-1)$-strongly monotone. 
Now, for every $x\in \Hi,$ we define a set-valued mapping $K:\Hi\rightrightarrows \Hi$ by
$$K(x)=K((x_k)_{k=0}^{\infty})=\Big\{(y_k)_{k=0}^{\infty}|~y_0\geq \frac{x_0}{10},~\mbox{and}~y_k=0,~\forall k\in \N\Big\}.$$
It is easy to verify that $K(x)$ is a closed and convex set for any $x\in \Hi$ (see \cite{keten2021efficient}). Let us mention that $P_{K(x)}:\Hi\rightarrow K(x)$ is a metric projection given by
$$P_{K(x)}(z_0,z_1,z_2,\dots)=\left\{
\begin{array}{ll}
   (z_0,z_1,z_2,\dots)  ,&\mbox{if}~ (z_0,z_1,z_2,\dots)\in K(x)\\
  \Big(\frac{x_0}{10},0,0,0,\dots)  ,&\mbox{if}~ (z_0,z_1,z_2,\dots)\notin K(x)~\mbox{and}~z_0<\frac{x_0}{10}\\
  \Big(z_0,0,0,0,\dots)  ,&\mbox{if}~ (z_0,z_1,z_2,\dots)\notin K(x)~\mbox{and}~z_0\geq \frac{x_0}{10}.
\end{array}
\right.$$
For any $x=(x_k)_{k=0}^{\infty},~y=(y_k)_{k=0}^{\infty}$ and $z=(z_k)_{k=0}^{\infty},$ it was shown that
$$\|P_{K(x)}(z)-P_{K(y)}(z)\|_2\leq \frac{1}{10}\|x-y\|_2.$$
Therefore, the condition \eqref{cond_proj1} is true with $l=\frac{1}{10}.$ Consequently, the problem has a unique solution which is $x^*=(0,0,0,\dots).$ 
\end{example}
Figure 1 shows convergence behavior of $\|x^k -x^*\|_2$ for the Tseng type algorithm with initial point $x^0=(\frac 12,\frac 16,\cdots)$, and $\alpha=2,~\lambda=0.1.$
\begin{center}
   \begin{figure}[!ht]
    \includegraphics[width=0.7\columnwidth]{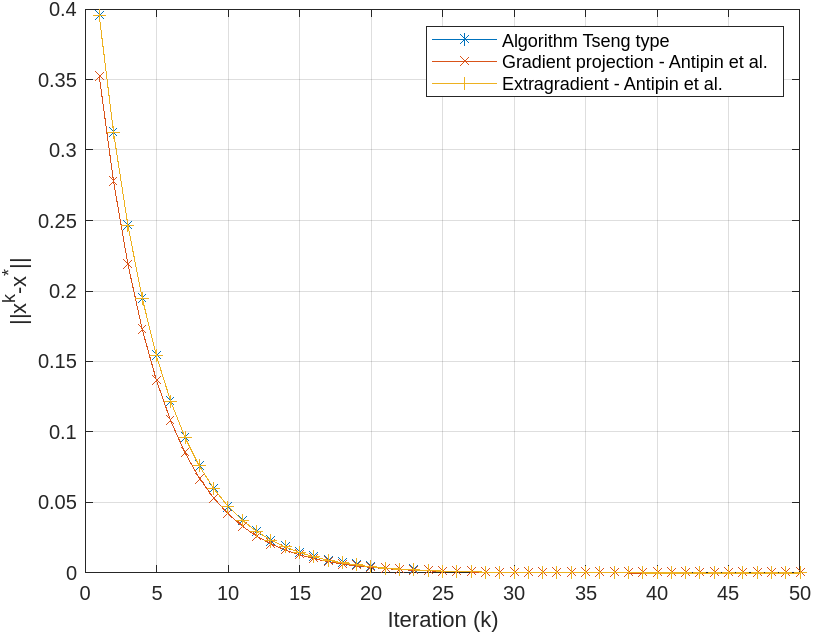}
    \caption{Convergence behavior of $\|x^k -x^*\|_2$ compared with algorithms in \cite{antipin2018extragradient} using MATLAB}
\end{figure} 
\end{center}
Before we consider the special case with the moving set, let us mention the following result.

\begin{lemma}
    Let $m:\Hi\rightarrow \Hi$ be a single-valued mapping, and $K\subseteq H$ be nonemplty closed convex subset. Then it holds
    \begin{eqnarray*}
        P_{m(x)+K}(z)=m(x)+P_K(z-m(x))~~\forall z\in \Hi.
    \end{eqnarray*}
\end{lemma}

\begin{rem} 
    If $K(x)=m(x)+K,$ then the $(QVI)$ is reduced to the case with the moving set. If $m:\Hi\rightarrow \Hi$ be $\beta$-Lipschitz continuous, with $\beta\geq 0,$ i.e.
     $$\|m(x)-m(y)\|\leq \beta \|x-y\|~~\forall x,y\in \Hi,$$
     then the condition is fulfilled with $l=2\beta.$ Indeed
     \begin{eqnarray*}
         &&\| P_{m(x)+K}(z)-P_{m(y)+K}(z)\|\leq \|m(x)+P_K(z-m(x))-m(y)-P_K(z-m(y)\|\\
         &&\leq \beta\|x-y\|+\|P_K(z-m(x))-P_K(z-m(y)\|\leq 2\beta\|x-y\|.
     \end{eqnarray*}
\end{rem}

\begin{prop}  Let $F:\Hi\rightarrow \Hi$ be $L$-Lipschitz continuous and $\rho$-strongly monotone on $\Hi.$ Let $K(x):=m(x)+K,$ and $m:\Hi\rightarrow \Hi$ be $\beta$-Lipschitz continuous. If
     $$(1+\theta)(1+\lambda L)<2\sqrt{1-\beta^2+\beta}-1,$$
then the sequence $\{x_n\}$ generated by the Tseng type algorithm converges linearly to the unique solution of $(QVI).$
\end{prop}

\section{Conclusion}
We have introduced a new approach for solving QVI based on Tseng's scheme (\cite{tseng2000modified}, also see \cite{boct2020forward}), which has been studied for variational inequality problems. A similar approach can be explored with weak assumptions such as strongly pseudo- monotonicity and parameters' assumption can be adjusted. These will be topics of future research. 

\subsection*{Acknowledgements} 
The author gratefully acknowledges the excellent research environment provided by Prof. R.I. Bo\c t, as well as E.R. Csetnek for their valuable discussions in improving the manuscript during his visit to the Faculty of Mathematics at the University of Vienna.

\end{document}